\documentclass[11pt,a4paper]{article}

\usepackage[hyperfootnotes=false,pdfpage layout=SinglePage]{hyperref}
\usepackage{mathptmx}
\usepackage[scaled=.92]{helvet}
\usepackage{amssymb}
\usepackage{amsmath}
\usepackage{amsthm}
\usepackage{latexsym}
\usepackage{amsfonts}
\usepackage{mathrsfs}
\usepackage{setspace}
\usepackage{graphicx}

\newtheorem{thm}{Theorem}[section]
\newtheorem{cor}[thm]{Corollary}
\newtheorem{lem}[thm]{Lemma}
\newtheorem{prop}[thm]{Proposition}
\theoremstyle{definition}
\newtheorem{defn}[thm]{Definition}
\newtheorem{rem}[thm]{Remark}
\newtheorem{ex}[thm]{Example}

\numberwithin{equation}{section}
\newcommand{\R}{\ensuremath{\mathbb R}}    
\newcommand{\C}{\ensuremath{\mathbb C}}    
\newcommand{\N}{\ensuremath{\mathbb N}}    
\newcommand{\K}{\ensuremath{\mathbb K}}    

\newcommand{\lk}{\langle}
\newcommand{\rk}{\rangle}
\newcommand{\aproduct}{\langle\cdot\,,\cdot\rangle}

\newcommand{\calC}{\mathcal C}

\newcommand{\calH}{\mathcal H}

\newcommand{\calK}{\mathcal K}

\newcommand{\calO}{\mathcal O}

\newcommand{\veps}{\varepsilon}
\newcommand{\vphi}{\varphi}

\newcommand{\mat}[4]
{
   \begin{pmatrix}
      #1 & #2\\
      #3 & #4
   \end{pmatrix}
}
\newcommand{\vek}[2]
{
   \begin{pmatrix}
      #1\\
      #2
   \end{pmatrix}
}

\renewcommand{\ker}{\operatorname{ker}}
\DeclareMathOperator{\ran}{ran}
\DeclareMathOperator{\dom}{dom}

\newcommand{\Sra}{\Rightarrow}

\newcommand{\Slra}{\Leftrightarrow}

\newcommand{\ol}{\overline}

\newcommand{\wt}{\widetilde}

\renewcommand{\l}{\langle}
\renewcommand{\r}{\rangle}

\DeclareMathOperator{\diag}{diag}
\DeclareMathOperator{\tr}{tr}

\newcommand{\bitem}{\begin{itemize}}
\newcommand{\eitem}{\end{itemize}}
\newcommand{\benum}{\begin{enumerate}}
\newcommand{\eenum}{\end{enumerate}}
\newcommand{\beq}{\begin{equation}}
\newcommand{\eeq}{\end{equation}}

\begin{document}
\vspace*{-.3cm}
\begin{center}
\begin{spacing}{1.7}
{\LARGE\bf Scalable Frames}
\end{spacing}

\vspace{.5cm}
\begin{spacing}{1.7}
{\large Gitta Kutyniok, Kasso A.\ Okoudjou, Friedrich Philipp, and Elizabeth K.\ Tuley}
\end{spacing}
\end{center}

\vspace{.1cm}\hrule\vspace*{.4cm}
\noindent{\bf Abstract}

\vspace*{.3cm}\noindent
Tight frames can be characterized as those frames which possess optimal numerical stability
properties. In this paper, we consider the question of modifying a general frame to generate a tight frame by
rescaling its frame vectors; a process which can also be regarded as perfect preconditioning of a frame by
a diagonal operator. A frame is called scalable, if such a diagonal operator exists. We derive various
characterizations of scalable frames, thereby including the infinite-dimensional situation. Finally, we
provide a geometric interpretation of scalability in terms of conical surfaces.
\vspace*{.4cm}\hrule
%
%

\vspace*{.5cm}
\section{Introduction}

Frames have established themselves by now as a standard notion in applied mathematics, computer science, and engineering,
see \cite{ch,ck12}.
In contrast to orthonormal bases, typically frames form redundant systems, thereby allowing non-unique, but stable
decompositions and expansions. The wide range of applications of frames can be divided into two categories. One
type of applications utilize frames for decomposing data. Here typical goals are erasure-resilient transmission,
data analysis or processing, and compression -- the advantage of frames being their robustness as well as their
flexibility in design. A second type of applications requires frames for expanding data. This approach is extensively
used in sparsity methodologies such as Compressed Sensing (see \cite{ek}), but also, for instance, as systems
generating trial spaces for PDE solvers. Again, it relies on non-uniqueness of the expansion which promotes
sparse expansions and on the flexibility in design.

All such applications require the associated algorithms to be numerically stable, which the subclass of tight frames
satisfies optimally. Thus, a prominent question raised in several publications so far is the following: When can a
given frame be modified to become a tight frame? The simplest operation to imagine is
to rescale each frame vector. Therefore this question is typically phrased in the following more precise form:
When can the vectors of a given frame be rescaled to obtain a tight frame? This is the problem we shall address
in this paper.

\subsection{Tight Frames}

Let us first state the precise definition of a frame and, in particular, of tight and Parseval frames to stand on
solid ground for the subsequent discussion. Letting $\calH$ be a real or complex separable Hilbert space and letting
$J$ be a subset of $\N$, a set of vectors $\Phi = \{\vphi_j\}_{j\in J}\subset\calH$ is called a
{\it frame} for $\calH$, if there exist positive constants $A,B > 0$ (the {\em lower} and {\em upper frame bound})
such that
\begin{equation}\label{e:frame_def}
A\|x\|^2\,\le\,\sum_{j\in J}|\lk x,\vphi_j\rk|^2\,\le\,B\|x\|^2 \quad \mbox{for all } x\in\calH.
\end{equation}
A frame $\Phi$ is called {\it $A$-tight} or just {\it tight}, if $A = B$ is possible in \eqref{e:frame_def},
and {\em Parseval}, if $A = B = 1$ is possible. Moreover, if $|J| < \infty$ (which implies that $\calH = \K^N$
with $\K=\R$ or $\K=\C$), the frame $\Phi$ is called {\em finite}.

To justify the claim of numerical superiority of tight frames, let $\Phi = \{\vphi_j\}_{j\in J}\subset\calH$ be a frame for
$\calH$ and let $T_\Phi : \calH\to\ell^2(J)$ with $T_\Phi x := \big(\l x,\vphi_j\r\big)_{j\in J}$ denote the
associated {\em analysis operator}.
Its adjoint $T_\Phi^*$, the {\em synthesis operator}
of $\Phi$, maps $\ell^2(J)$ surjectively onto $\calH$. From the properties of $T_\Phi$, it follows that the {\em
frame operator} $S_\Phi := T_\Phi^*T_\Phi$ of $\Phi$, given by
\[
S_\Phi x = \sum_{j\in J}\l x,\vphi_j\r\vphi_j,\quad x\in\calH,
\]
is a bounded and strictly positive selfadjoint operator in $\calH$. These properties imply that $\Phi$ admits the
reconstruction formula
\[
x = \sum_{j\in J}\l x,\vphi_j\r S_\Phi^{-1}\vphi_j  \quad \mbox{for all } x\in\calH.
\]

Inversion requires particular numerical attention, which implies that $S_\Phi = const \cdot I_\calH$ is desirable
($I_\calH$ denoting the identity on $\calH$, for $\calH = \K^N$ we will use $I_N$).
And in fact, tight frames can be characterized as precisely those frames satisfying this
condition. Thus an $A$-tight frame admits the numerically optimally stable reconstruction given by
\[
x = A^{-1} \cdot \sum_{j\in J}\l x,\vphi_j\r \vphi_j  \quad \mbox{for all } x\in\calH.
\]


\subsection{Generating Parseval Frames}

This observation raises the question on how to carefully modify a given frame -- which might be suitable for a
particular application -- in order to generate a tight frame. It is immediate that this question is equivalent to
generating a Parseval frame provided we allow multiplication of each frame vector by the same value. Thus
typically one seeks to generate Parseval frames.

A very common approach is to apply $S_\Phi^{-1/2}$ to each frame vector of a frame $\Phi$, which can be easily
shown to yield a Parseval frame. This approach is though of more theoretical interest due to the repetition of
the problem to invert the frame operator. Hence, this construction is often not reasonable in practice.

The simplest imaginable variation of a frame is just scaling its frame vectors. We thus coin a frame {\it scalable},
if such a scaling leads to a Parseval frame. It should also be pointed out that the scaling of frames is related
to the notion of signed frames, weighted frames as well as controlled frames (see, e.g., \cite{pewal02,bag,raba10}).

It is evident that not every frame is scalable. For example, a basis in $\R^2$ which is not an orthogonal basis
is not scalable, since a frame with two elements in $\R^2$ is a Parseval frame if and only if it is an orthonormal
basis. As a first classification, the finite-dimensional version of Proposition \ref{p:1st_charac} shows
that a frame $\Phi$ in $\K^N$ with analysis operator $T_\Phi$
(the rows of which are the frame vectors) is scalable if and only if there exists a diagonal matrix $D$ such that
$DT_\Phi$ is isometric. Since the condition number of such a matrix equals one, the scaling question is a particular
instance of the problem of preconditioning of matrices.


\subsection{An Excursion to Numerical Linear Algebra}

In the numerical linear algebra community, the problem of preconditioning is well-known and extensively studied,
see, e.g., \cite{c,fk}. The problem to design preconditioners involving scaling appears in various forms in the
numerical linear algebra literature. The common approach to this problem is to minimize the condition number of
the matrix multiplied by a preconditioning matrix -- in our case of $DT_\Phi$, where $D$ runs through the set
of diagonal matrices. As shown for instance in \cite{bm}, this minimization problem can be reformulated as a
convex problem. However, as also mentioned in \cite{bm}, algorithms solving this convex problem perform slowly,
and, even worse, there exist situations in which the infimum is not attained. As additional references,
we wish to mention \cite{be,bb,c,k,s1} for preconditioning by multiplying
diagonal matrices from the left and/or the right, \cite{s2,els,e} for block diagonal scaling and \cite{r,b,vS}
for scaling in order to obtain equal-norm rows or columns.


\subsection{Our Contribution}

Our contribution to the scaling problem of frames is three-fold. First, these are the leadoff results on
this problem. Second, with Theorem \ref{t:main} we provide various characterizations of (strict) scalability
of a frame for a general separable Hilbert space. In this respect, a particular interesting characterization derived in
Theorem \ref{t:main} states that a frame $\Phi$ in a Hilbert space $\calH$ is strictly scalable if and only if
there exists a frame $\Psi$ in a presumably different Hilbert space $\calK$ such that the coupling of the frame vectors of
$\Phi$ and $\Psi$ in $\calH\oplus\calK$ constitutes an orthogonal basis. And, third, Theorems \ref{t:quadric} and \ref{t:geometry}
provide a geometric characterization of scalability of finite frames. More precisely, we prove that a finite frame in $\R^N$
is not scalable if and only if all its frame vectors are contained in certain cones.


\subsection{Outline}

This paper is organized as follows. In Section \ref{s:char} we focus on the situation of general separable
Hilbert spaces. We first analyze when a scaling preserves the frame property (Subsection \ref{subsec:frameproperties}),
followed by a general equivalent condition in terms of diagonal operators (Subsection \ref{subsec:generalequivalentcondition}).
Subsection \ref{subsec:mainresult} is devoted to the main characterization of strict scalability of
frames. In Section \ref{s:real} we then restrict to the situation of finite frames. First, in Subsection
\ref{subsec:characterization}, we derive a yet different characterization tailored specifically to the
finite-dimensional case. Finally, this result is shown to give rise to a geometric interpretation of scalable frames
in terms of quadrics (Subsection \ref{subsec:geo}).


\section{Strict Scalability of General Frames}\label{s:char}

In this section, we derive our first main theorem which provides a characterization of (strictly) scalable frames.
We wish to mention that this result does not only hold for finite frames, but in the general separable Hilbert space setting.


\subsection{Scalability and Frame Properties}\label{subsec:frameproperties}

We start by making the notion of scalability mathematically precise. We further introduce the notions of
positive and strict scalability. Positive scalability ensures that no frame vectors are suppressed by the
preconditioning. The same is true for strict scalability, which in addition prevents numerical instabilities
caused by arbitrarily small entries in the matrix representation of the diagonal operator serving as preconditioner.

\begin{defn}
A frame $\Phi = \{\vphi_j\}_{j\in J}$ for $\calH$ is called {\em scalable} if there exist scalars $c_j\ge 0$, $j\in J$, such
that $\{c_j\vphi_j\}_{j\in J}$ is a Parseval frame. If, in addition, $c_j > 0$ for all $j\in J$, then $\Phi$ is called
{\em positively scalable}. If there exists $\delta > 0$, such that $c_j\ge\delta$ for all $j\in J$, then $\Phi$ is called
{\em strictly scalable}.
\end{defn}

Clearly, positive and strict scalability coincide for finite frames. Moreover, each scaling $\{c_j\vphi_j\}_{j\in J}$ of a finite
frame $\{\vphi_j\}_{j\in J}$ with positive scalars $c_j$ is again a frame. In the infinite-dimensional situation this might
not be the case. However, if there exist $K_1,K_2 > 0$ such that $K_1\le c_j\le K_2$ holds for all $j\in J$, then also
$\{c_j\vphi_j\}_{j\in J}$ is a frame, see \cite[Lemma 4.3]{bag}. A characterization of when a scaling preserves the frame
property can be found in Proposition \ref{p:diag_frame} below. This requires particular attention to the {\em diagonal operator}
$D_c$ in $\ell^2(J)$ corresponding to a sequence $c = (c_j)_{j\in J}\subset\K$, which is defined by
$$
D_c(v_j)_{j\in J} := \big(c_jv_j\big)_{j\in J}\,,\quad (v_j)_{j\in J}\in\dom D_c,
$$
where
$$
\dom D_c := \left\{(v_j)_{j\in J}\in\ell^2(J) : (c_jv_j)_{j\in J}\in\ell^2(J)\right\}.
$$
It is well-known that $D_c$ is a (possibly unbounded) selfadjoint operator in $\ell^2(J)$ if and only if $c_j\in\R$ for all
$j\in J$. If even $c_j\ge 0$ ($c_j > 0$, $c_j\ge\delta > 0$) for each $j\in J$, then the selfadjoint operator $D_c$ is
non-negative (positive, strictly positive, respectively).

Before we present the announced characterization, we require some notation. As usual, we denote the
domain, the kernel and the range of a linear operator $T$ by $\dom T$, $\ker T$ and $\ran T$, respectively. Also, a closed
linear operator $T$ between two Hilbert spaces $\calH$ and $\calK$ will be called {\em ICR} {\rm (}or an {\em ICR-operator}{\rm)},
if it is injective and has a closed range, i.e., if there exists $\delta > 0$ such that $\|Tx\|\ge\delta\|x\|$ for all $x\in\dom T$.
We mention that the analysis operator of a frame is always an ICR-operator.

The following result now provides a characterization of when a scaling preserves the frame property.

\begin{prop}\label{p:diag_frame}
Let $\Phi = \{\vphi_j\}_{j\in J}$ be a frame for $\calH$ with analysis operator $T_\Phi$ and let $c = (c_j)_{j\in J}$ be a
sequence of non-negative scalars. Then the following conditions are equivalent.
\bitem
\item[{\rm (i)}]  The scaled sequence of vectors $\Psi := \{c_j\vphi_j\}_{j\in J}$ is a frame for $\calH$.
\item[{\rm (ii)}] We have $\ran T_\Phi\subset\dom D_c$ and $D_c|\ran T_\Phi$ is ICR.
\eitem
Moreover, in this case, the frame operator of the frame $\Psi$ is given by
$$
S_\Psi = (D_cT_\Phi)^*(D_cT_\Phi) = \ol{T_\Phi^*D_c}D_cT_\Phi,
$$
where $\ol{T_\Phi^*D_c}$ denotes the closure of the operator $T_\Phi^*D_c$.
\end{prop}

\begin{proof}
(i)$\Sra$(ii). Assume that $\Psi$ is a frame and denote its analysis operator by $T_\Psi$. Then, for $x\in\calH$, the
$j$-th component of $T_\Psi x$ is given by
$$
(T_\Psi x)_j = \l x,c_j\vphi_j\r = c_j\l x,\vphi_j\r = (D_cT_\Phi x)_j.
$$
Hence, $T_\Psi = D_cT_\Phi$. As $\dom T_\Psi = \calH$, this implies $\ran T_\Phi\subset\dom D_c$. Since $\Phi$ is a frame,
$\ran T_\Phi$ is a closed subspace. And since $\Psi$ is a frame, there exist $A',B' > 0$ such that
$A'\|x\|^2\le\|D_cT_\Phi x\|_2^2\le B'\|x\|^2$ for all $x\in\calH$. In particular, for $v = T_\Phi x\in\ran T_\Phi$ we have
$$
\|D_c v\|_2^2 = \|D_cT_\Phi x\|_2^2\ge A'\|x\|^2\ge A'\|T_\Phi\|^{-2}\|v\|_2^2,
$$
which shows that $D_c|\ran T_\Phi$ is an ICR-operator.

(ii)$\Sra$(i). Conversely, assume that $\ran T_\Phi\subset\dom D_c$ and that $D_c|\ran T_\Phi$ is ICR. By the closed graph
theorem and $\ran T_\Phi\subset\dom D_c$, the operator $D_c|\ran T_\Phi$ is bounded, which implies the existence of $A',B' > 0$ such that
$$
A'\|v\|_2^2 \le \|D_c v\|_2^2 \le B'\|v\|_2^2
$$
holds for all $v\in\ran T_\Phi$. Setting $v = T_\Phi x$ and noting that $T_\Phi$ is bounded and ICR, we obtain constants $A'',B'' > 0$ such that
$$
A''\|x\|^2 \le \|D_c T_\Phi x\|_2^2 \le B''\|x\|^2
$$
holds for all $x\in\calH$. Consequently, $\Psi$ is a frame.

It remains to prove the {\em moreover}-part, i.e., that $(D_c T_\Phi)^* = \ol{T_\Phi^*D_c}$. Since $D_cT_\Phi$ is bounded, so is its adjoint
$(D_cT_\Phi)^*$. In addition, it is easy to see that $T_\Phi^* D_c v = (D_cT_\Phi)^*v$ holds for all $v$ in the dense subspace $\dom D_c$.
Hence, $T_\Phi^* D_c$ is bounded and densely defined. Its bounded closure thus coincides with $(D_cT_\Phi)^*$.
\end{proof}

It is evident that the operator $D_c$ in Proposition \ref{p:diag_frame} is in general unbounded. The following corollary provides a condition on the frame $\Phi$ which leads to necessarily bounded diagonal operators $D_c$ in Proposition \ref{p:diag_frame}. We remark that $\liminf_{j\in J}$ shall be interpreted as $\liminf_{j\in J,\,j\to\infty}$, which is a proper definition, since $J\subset\N$ was assumed. As it is custom, we set $\liminf_{j\in J}$ to $\infty$ if $J$ is finite.

\begin{cor}\label{c:diag_frame}
Let $\Phi$, $\Psi$ and $c$ be as in Proposition {\rm\ref{p:diag_frame}} and assume $\liminf_{j\in J}\|\vphi_j\| > 0$.
Then $\Psi$ is a frame if and only if $D_c$ is bounded and $D_c|\ran T_\Phi$ is ICR. In this case, we have
$$
S_\Psi = (D_cT_\Phi)^*(D_cT_\Phi) = T_\Phi^*D_c^2T_\Phi.
$$
\end{cor}

\begin{proof}
If $D_c$ has the above-mentioned properties, then $\Psi$ is a frame by Proposition \ref{p:diag_frame}. If $\Psi$ is a
frame, then there exists $B > 0$ such that for each $x\in\calH$ we have
$$
\sum_{j\in J}c_j^2|\l x,\vphi_j\r|^2\,\le\,B\|x\|^2.
$$
In particular, for $k\in J$, $c_k^2\|\vphi_k\|^4\le B\|\vphi_k\|^2$. Since there exist $\delta > 0$ and $j_0\in J$ such that
$\|\vphi_j\|\ge\delta$ for all $j\in J$, $j\ge j_0$, this implies $c_k\le B^{1/2}\delta^{-1}$ for all $k\in J$, $k\ge j_0$.
Thus $D_c$ is bounded as $\|D_c\| = \sup_{j\in J}c_k$.
\end{proof}


\subsection{General Equivalent Condition}
\label{subsec:generalequivalentcondition}

We now state a seemingly obvious equivalent condition to scalability, which is however not straightforward to
state and prove in the general setting of an arbitrary separable Hilbert space.

\begin{prop}\label{p:1st_charac}
Let $\Phi = \{\vphi_j\}_{j\in J}$ be a frame for $\calH$. Then the following conditions are equivalent.
\bitem
\item[{\rm (i)}] $\Phi$ is {\rm (}positively, strictly{\rm )} scalable.
\item[{\rm (ii)}] There exists a non-negative {\rm (}positive, strictly positive, respectively{\rm )} diagonal operator $D$ in $\ell^2(J)$
such that
\begin{equation}\label{e:1st_charac}
\ol{T_\Phi^*D}DT_\Phi = I_\calH.
\end{equation}
\eitem
\end{prop}

\begin{proof}
(i)$\Sra$(ii). If $\Phi$ is scalable with a sequence of non-negative scalars $(c_j)_{j\in J}$, then $\Psi := \{c_j\vphi_j\}_{j\in J}$ is a
Parseval frame. In particular, $\Psi$ is a frame, which, by Proposition \ref{p:diag_frame}, implies that
$\ran T_\Phi\subset\dom D_c$ and that $S_\Psi = \ol{T_\Phi^*D_c}D_cT_\Phi$ is the frame operator of $\Psi$. Since the frame
operator of a Parseval frame coincides with the identity operator, it follows that $\ol{T_\Phi^*D_c}D_cT_\Phi = I_\calH$.

(ii)$\Sra$(i). Conversely, assume that there exists a non-negative diagonal operator $D$ in $\ell^2(J)$ such that $\ol{T_\Phi^*D}DT_\Phi = I_\calH$.
Then $DT_\Phi$ is everywhere defined. In particular, this implies that $\ran T_\Phi\subset\dom D$. Since $T_\Phi$ is bounded and $D$ is
closed, the operator $DT_\Phi$ is closed. Hence, by the closed graph theorem, $DT_\Phi$ is a bounded operator from $\calH$ into $\ell^2(J)$. In fact,
$(DT_\Phi)^*(DT_\Phi) = I_\calH$ implies that $DT_\Phi$ is even isometric. Thus, from the boundedness of $T_\Phi$ we conclude that
$D|\ran T_\Phi$ is ICR. Let $c = (c_j)_{j\in J}$ be the sequence of non-negative scalars such that $D = D_c$. As a consequence
of Proposition \ref{p:diag_frame}, $\Psi := \{c_j\vphi_j\}_{j\in J}$ is a frame with frame operator $S_\Psi = I_\calH$, which implies
that $\Psi$ is a Parseval frame.

The proofs for positive and strict scalability of $\Phi$ follow analogous lines.
\end{proof}

Under certain assumptions, the relation \eqref{e:1st_charac} can be simplified as stated in the following remark which directly follows
from Corollary \ref{c:diag_frame}.

\begin{rem}\label{r:bounded}
If $\delta := \liminf_{j\in J}\|\vphi_j\| > 0$, then a diagonal operator $D$ as in Proposition \ref{p:1st_charac} is necessarily
bounded, and \eqref{e:1st_charac} reads
\[
T_\Phi^*D^2T_\Phi = I_\calH.
\]
\end{rem}

Before stating our main theorem in this section, we first provide a highly useful implication of Proposition \ref{p:1st_charac},
showing that scalability is invariant under unitary transformations.

\begin{cor}\label{c:unitary}
Let $U$ be a unitary operator in $\calH$. Then a frame $\Phi = \{\vphi_j\}_{j\in J}$ for $\calH$ is scalable if and only if the
frame $U\Phi = \{U\vphi_j\}_{j\in J}$ is scalable.
\end{cor}

\begin{proof}
Let $\Phi$ be a scalable frame for $\calH$ with diagonal operator $D$. Since the analysis operator of $U\Phi$ is given by $T_{U\Phi} = T_\Phi U^*$,
\[
\ol{T_{U\Phi}^*D}DT_{U\Phi} = \ol{UT_\Phi^*D}DT_\Phi U^* = U\ol{T_\Phi^*D}DT_\Phi U^* = UU^* = I_\calH,
\]
which implies scalability of $U\Phi$.

The converse direction can be proved similarly.
\end{proof}


\subsection{Main Result}
\label{subsec:mainresult}

To state the main result of this section, we require the notion of an orthogonal basis, which we recall for the convenience of the reader.
A sequence $\{v_k\}_k$ of non-zero vectors in a Hilbert space $\calK$ is called an {\em orthogonal basis} of $\calK$, if
$\inf_k\|v_k\| > 0$ and $(v_k/\|v_k\|)_k$ is an orthonormal basis of $\calK$.

The following result provides several equivalent conditions for a frame $\Phi$ to be strictly scalable. We are already familiar with
condition (ii). Condition (iii) can be interpreted as a `diagonalization' of the Grammian of $\Phi$, and condition (iv) shows
that $\Phi$ can be orthogonally expanded to an orthogonal basis.

\begin{thm}\label{t:main}
Let $\Phi = \{\vphi_j\}_{j\in J}$ be a frame for $\calH$ such that $\liminf_{j\in J}\|\vphi_j\| > 0$, and let $T = T_\Phi$ denote its
analysis operator. Then the following statements are equivalent.
\begin{enumerate}
\item[{\rm (i)}]   The frame $\Phi$ is strictly scalable.
\item[{\rm (ii)}]  There exists a strictly positive bounded diagonal operator $D$ in $\ell^2(J)$ such that $DT$ is isometric {\rm (}that is,
$T^*D^2T = I_\calH${\rm )}.
\item[{\rm (iii)}] There exist a Hilbert space $\calK$ and a bounded ICR operator $L : \calK\to\ell^2(J)$ such that $TT^* + LL^*$ is a
strictly positive bounded diagonal operator.
\item[{\rm (iv)}]  There exist a Hilbert space $\calK$ and a frame $\Psi = \{\psi_j\}_{j\in J}$ for $\calK$ such that the vectors
$$
\vphi_j\oplus\psi_j\in\calH\oplus\calK,\quad j\in J,
$$
form an orthogonal basis of $\calH\oplus\calK$.
\end{enumerate}
If one of the above conditions holds, then the frame $\Psi$ from {\rm (iv)} is strictly scalable, its analysis operator is given by an operator $L$ from {\rm (iii)}, and with a diagonal operator $D$ from {\rm (ii)} we have
\begin{equation}\label{e:zudem}
L^*D^2L = I_{\calK},\quad\text{and}\quad L^*D^2T = 0.
\end{equation}
\end{thm}

\begin{proof}
(i)$\Slra$(ii). This equivalence follows from Proposition \ref{p:1st_charac} (see also Remark \ref{r:bounded}).

(ii)$\Slra$(iii). For the proof of (ii)$\Sra$(iii) let $D$ be a strictly positive bounded diagonal operator in $\ell^2(J)$
such that $T^*D^2T = I_\calH$. For the Hilbert space $\calK$ in (iii) we choose $\calK := (\ran DT)^\perp = \ker T^*D\subset\ell^2(J)$. On $\calK$
we define the operator $L : \calK\to\ell^2(J)$ by $L := D^{-1}|\calK$, which clearly is a bounded ICR operator. Then $L^* = P_\calK D^{-1}$,
where $P_\calK$ denotes the orthogonal projection in $\ell^2(J)$ onto $\calK$. Let us show that $DTT^*D + P_\calK$ coincides with the
identity operator on $\ell^2(J)$. Then
\begin{align*}
TT^* + LL^*
&= D^{-1}DTT^*DD^{-1} + D^{-1}P_\calK D^{-1}\\
&= D^{-1}\big(DTT^*D + P_\calK\big)D^{-1}\\
&= D^{-2},
\end{align*}
which is a strictly positive bounded diagonal operator in $\ell^2(J)$, and (iii) is proved. Since $DT$ is isometric, we have
$$
\big(DTT^*D\big)^2 = DT(DT)^*(DT)T^*D = DTT^*D,
$$
which shows that $DTT^*D$ is a projection. Moreover, $DTT^*D$ is selfadjoint and thus an orthogonal projection. Since its kernel coincides with
$\ker T^*D = \calK$, it is the orthogonal projection onto $\calK^\perp$. This shows that $DTT^*D + P_\calK = I_{\ell^2(J)}$.

To prove the converse implication, suppose that (iii) holds with a Hilbert space $\calK$ and a bounded ICR operator $L : \calK\to\ell(J)$,
such that $TT^* + LL^* = D^{-2}$ with a strictly positive bounded diagonal operator $D$. Note that also $D^{-1}$ is strictly positive and
bounded. Define the operator
\begin{equation}\label{e:G}
G : \calH\oplus\calK\to\ell^2(J),\quad G\vek x y := Tx + Ly,\;\;\vek x y\in\calH\oplus\calK.
\end{equation}
Then $G^*v = (T^*v,L^*v)^T$, $v\in\ell^2(J)$, and hence $GG^* = TT^* + LL^* = D^{-2}$. In particular, $G$ is an isomorphism between
$\calH\oplus\calK$ and $\ell^2(J)$. Moreover, we have
$$
G^*D^2G = G^*D^2D^{-2}G^{-*} = I_{\calH\oplus\calK}.
$$
This implies that
$$
\mat{I_\calH}00{I_\calK} = \vek{T^*}{L^*}(D^2T,D^2L) = \mat{T^*D^2T}{T^*D^2L}{L^*D^2T}{L^*D^2L},
$$
or, equivalently,
$$
T^*D^{2}T = I_\calH, \quad L^*D^2L = I_\calK,\quad\text{and}\quad L^*D^2T = 0,
$$
which, in particular, yields (ii) (and \eqref{e:zudem}).

(iii)$\Slra$(iv). For the implication (iii)$\Sra$(iv), let $D$, $L$ and $G$ be as above and define $\psi_j := L^*e_j$, $j\in J$, where
$e_j$ denotes the $j$-th vector of the standard orthonormal basis $\{e_j\}_{j\in J}$ of $\ell^2(J)$. As $L$ is a
bounded ICR operator and
$$
\sum_{j\in J}|\l x,\psi_j\r|^2 = \sum_{j\in J}|\l x,L^*e_j\r|^2 = \sum_{j\in J}|\l Lx,e_j\r|^2 = \|Lx\|^2
$$
for all $x\in\calK$, it follows that $\Psi = \{\psi_j\}_{j\in J}$ is a frame. Note that $T^*e_j = \vphi_j$, $j\in J$.
Hence, $\vphi_j\oplus\psi_j = T^*e_j \oplus L^*e_j = G^*e_j$, $j\in J$, and therefore
$$
\l\vphi_j\oplus\psi_j,\vphi_k\oplus\psi_k\r = \l G^*e_j,G^*e_k\r = \l GG^*e_j,e_k\r = \l D^{-2}e_j,e_k\r = c_j^{-2}\delta_{jk}.
$$
As the $c_j$'s are bounded and $G^*$ is an isomorphism, this shows that the sequence $\{\vphi_j\oplus\psi_j\}_{j\in\ J}$ is an orthogonal basis of $\ell^2(J)$.

Finally, to prove the converse implication, suppose that (iv) holds true and denote by $L$ the analysis operator of the frame $\Psi$. Since
$\{\vphi_j\oplus\psi_j\}_{j\in J}$ is an orthogonal basis of $\calH\oplus\calK$, for all $j,k\in J$ we have
$\l\vphi_j,\vphi_k\r + \l\psi_j,\psi_k\r = d_j\delta_{jk}$, where $d_j = \|\vphi_j\|^2 + \|\psi_j\|^2$, $j\in J$.
Note that the sequence $(d_j)_{j\in J}$ is bounded and bounded from below by a positive constant. Hence, for all $j,k\in J$,
\begin{align*}
\l(TT^* + LL^*)e_j,e_k\r
&= \l T^*e_j,T^*e_k\r + \l L^*e_j,L^*e_k\r = \l\vphi_j,\vphi_k\r + \l\psi_j,\psi_k\r\\
&= d_j\delta_{jk} = \l d_je_j,e_k\r.
\end{align*}
This implies $TT^* + LL^* = D_d$, where $d := (d_j)_{j\in J}$. The operator $D_d$ is a strictly positive bounded diagonal operator,
which proves (iii).
\end{proof}

The restriction of conditions (iii) and (iv) in Theorem \ref{t:main} to the situation of finite frames is not immediate and requires
some thought. This is the focus of the next result.

\begin{cor}\label{r:finite}
Let $\Phi = \{\vphi_j\}_{j=1}^M$ be a frame for $\K^N$ and let $T = T_\Phi\in\K^{M\times N}$ denote the matrix representation of its analysis operator. Then the following statements are equivalent.
\begin{enumerate}
\item[{\rm (i)}]   The frame $\Phi$ is strictly scalable.
\item[{\rm (ii)}]  There exists a positive definite diagonal matrix $D\in\K^{M\times M}$ such that $DT$ is isometric.
\item[{\rm (iii)}] There exists $L\in\K^{M\times (M - N)}$ such that $TT^* + LL^*$ is a positive definite diagonal matrix.
\item[{\rm (iv)}]  There exists a frame $\Psi = \{\psi_j\}_{j=1}^M$ for $\K^{M-N}$ such that
$\{\vphi_j \oplus \psi_j\}_{j=1}^M \in \K^{M}$ forms an orthogonal basis of $\K^{M}$.
\end{enumerate}
\end{cor}

\begin{proof}
We prove this result by using the equivalent conditions from Theorem \ref{t:main}. First of all, we observe that $\calH = \K^N$ and
$\ell^2(J) = \K^M$. Moreover, condition (ii) obviously coincides with Theorem \ref{t:main}(ii), so that (i)$\Slra$(ii) holds. The
equivalence (iii)$\Slra$(iv) can be shown in a similar way as the equivalence (iii)$\Slra$(iv) in Theorem \ref{t:main}.

Hence, it remains to show that (iii) and the condition (iii) in Theorem \ref{t:main} are equivalent. For this, assume that (iii) holds,
set $\calK := \K^{M - N}$ and $G := [T|L]\in\K^{M\times M}$. Then, since $GG^* = TT^* + LL^*$ is a positive
definite diagonal matrix, it follows that $G$ is non-singular and therefore $\ker L = \{0\}$. Thus, $L$ is ICR, and (iii) in Theorem
\ref{t:main} holds.
For the converse, recall that the operator $G : \K^N\oplus\calK\to\K^M$ in \eqref{e:G} was shown to be an isomorphism in the proof of
Theorem \ref{t:main}. Hence, $\dim\calK = M - N$. Thus, with some (bijective) isometry $V : \K^{M-N}\to\calK$ and
$\wt L := LV\in\K^{M\times (M - N)}$ we have $TT^* + \wt L\wt L^* = TT^* + LL^*$.
\end{proof}

Finally, we apply Theorem \ref{t:main} to the special case of finite frames with $N+1$ frame vectors in $\K^N$, which leads
to a quite easily checkable condition for scalability. For this, we again require some prerequisites. Letting $\Phi = \{\vphi_i\}_{j=1}^M$
be a frame for the Hilbert space $\K^N$, by $\calO_\Phi$ we denote the set of indices $k\in \{1,\ldots,M\}$ for which $\l\vphi_k,\vphi_j\r = 0$
holds for all $j\in \{1,\ldots,M\}\setminus\{k\}$. Note that $\calO_\Phi = \{1,\ldots,M\}$ holds if and only if $\Phi$ is an orthogonal basis
of $\K^N$. In particular, this implies $M=N$.

\begin{cor}\label{t:N+1}
Let $\Phi = \{\vphi_j\}_{j=1}^{N+1}$ be a frame for $\K^N$ such that $\vphi_j\neq 0$ for all $j=1,\ldots,N+1$. Then
$\calO_\Phi\neq\{1,\ldots,N+1\}$, and the following statements are equivalent.
\begin{enumerate}
\item[{\rm (i)}]   $\Phi$ is strictly scalable.
\item[{\rm (ii)}]  There exist $k\in\{1,\ldots,N+1\}\setminus\calO_\Phi$ and $c > 0$ such that
$$
\l\vphi_i,\vphi_k\r\l\vphi_k,\vphi_j\r = -c\l\vphi_i,\vphi_j\r
$$
holds for all $i,j\in\{1,\ldots,N+1\}\setminus\{k\}$, $i\neq j$.
\item[{\rm (iii)}] For all $k\in\{1,\ldots,N+1\}\setminus\calO_\Phi$ there exists $c_k > 0$ such that
$$
\l\vphi_i,\vphi_k\r\l\vphi_k,\vphi_j\r = -c_k\l\vphi_i,\vphi_j\r
$$
holds for all $i,j\in\{1,\ldots,N+1\}\setminus\{k\}$, $i\neq j$.
\end{enumerate}
\end{cor}

\begin{proof}
As remarked before, $\calO_\Phi = \{1,\ldots,N+1\}$ implies that $\Phi$ is an orthogonal basis of $\K^N$, which is impossible.

(i)$\Sra$(iii). For this, let $k\in\{1,\ldots,N+1\}\setminus\calO_\Phi$ be arbitrary. By Theorem \ref{t:main} (see also
Corollary \ref{r:finite}) there exists $v = (v_1,\ldots,v_{N+1})^T\in\K^{N+1}$ such that $T_\Phi T_\Phi^* + vv^*$ is a diagonal
matrix. Hence, $\l\vphi_i,\vphi_j\r + v_i\ol{v_j} = 0$ holds for all $i,j\in\{1,\ldots,N+1\}$, $i\neq j$. Therefore, for
$i,j\in\{1,\ldots,N+1\}\setminus\{k\}$, $i\neq j$, we have
$$
\l\vphi_i,\vphi_k\r\l\vphi_k,\vphi_j\r = v_i\ol{v_j}|v_k|^2 = -|v_k|^2 \l\vphi_i,\vphi_j\r.
$$
If $v_k = 0$, then $\l\vphi_k,\vphi_j\r = 0$ for all $j\in\{1,\ldots,N+1\}\setminus\{k\}$. But since $k\notin\calO_\Phi$ was assumed,
it follows that $|v_k|^2 > 0$, and (iii) holds.

(iii)$\Sra$(ii). This is obvious.

(ii)$\Sra$(i). Assume now that (ii) is satisfied, and set
$$
v_k := \sqrt{c}\quad\text{and}\quad v_j := -v_k^{-1}\l\vphi_j,\vphi_k\r\quad(j\in\{1,\ldots,N+1\}\setminus\{k\}).
$$
Then $v_i\ol{v_k} = -\l\vphi_i,\vphi_k\r$ for $i\in\{1,\ldots,N+1\}\setminus\{k\}$ and
$$
v_i\ol{v_j} = |v_k|^{-2}\l\vphi_i,\vphi_k\r\l\vphi_k,\vphi_j\r = -\l\vphi_i,\vphi_j\r
$$
for $i,j\in\{1,\ldots,N+1\}\setminus\{k\}$, $i\neq j$. This implies that $T_\Phi T_\Phi^* + vv^*$ is a diagonal matrix whose diagonal
entries are positive (since otherwise $0 = \|\vphi_j\|^2 + |v_j|^2$ and thus $\vphi_j = 0$ for some $j\in\{1,\ldots,N+1\}$). Now,
(i) follows from Theorem \ref{t:main}.
\end{proof}

As mentioned above, Corollary \ref{t:N+1} might be utilized to test whether a frame for $\K^N$ with $N+1$ frame vectors is strictly
scalable or not. Such a test would consist of finding an index $k\notin\calO_\Phi$ and checking whether there exists a $c > 0$
such that $\l\vphi_i,\vphi_k\r\l\vphi_k,\vphi_j\r = -c\l\vphi_i,\vphi_j\r$ holds for all $i,j\in\{1,\ldots,N+1\}\setminus\{k\}$,
$i\neq j$.


\section{Scalability of Real Finite Frames}\label{s:real}

We next aim for a more geometric characterization of scalability. For this, we now focus on frames for $\R^N$.
The reason why we restrict ourselves to real frames is that in the proof of the main theorem in this section we
make use of the following variant of Farkas' Lemma which only exists for real vector spaces.

\begin{lem}\label{l:farkas}
Let $A : V\to W$ be a linear mapping between finite-dimensio\-nal real Hilbert spaces $(V,\aproduct_V)$ and $(W,\aproduct_W)$,
let $\{e_i\}_{i=1}^N$ be an orthonormal basis of $V$ and let $b\in W$. Then exactly one of the following statements holds:
\begin{enumerate}
\item[{\rm (i)}]  There exists $x\in V$ such that $Ax = b$ and $\lk x,e_i\rk_V\ge 0$ for all $i=1,\ldots,N$.
\item[{\rm (ii)}] There exists $y\in W$ such that $\lk b,y\rk_W < 0$ and $\lk Ae_i,y\rk_W\ge 0$ for all $i=1,\ldots,N$.
\end{enumerate}
\end{lem}

Lemma \ref{l:farkas} can be proved in complete analogy to the classical Farkas' Lemma, where $V = \R^n$ and $W = \R^m$, $n,m\in\N$.
A proof of this statement can, for instance, be found in \cite[Thm 5.1]{ber}.


\subsection{Characterization Result}
\label{subsec:characterization}

The following theorem provides a characterization of non-scalability of a finite frame specifically tailored to the
finite-dimensional case. In Subsection \ref{subsec:geo}, condition (iii) will then be utilized to derive an illuminating
geometric interpretation.

\begin{thm}\label{t:quadric}
Let $\Phi = \{\vphi_j\}_{j=1}^M\subset\R^N\setminus\{0\}$ be a frame for $\R^N$. Then the following statements are equivalent.
\begin{enumerate}
\item[{\rm (i)}]   $\Phi$ is not scalable.
\item[{\rm (ii)}]  There exists a symmetric matrix $Y\in\R^{N\times N}$ with $\tr(Y) < 0$ such that $\vphi_j^TY\vphi_j\ge 0$ for all $j = 1,\ldots,M$.
\item[{\rm (iii)}] There exists a symmetric matrix $Y\in\R^{N\times N}$ with $\tr(Y) = 0$ such that $\vphi_j^TY\vphi_j > 0$ for all $j = 1,\ldots,M$.
\end{enumerate}
\end{thm}

\begin{proof}
(i)$\Slra$(ii). Let $W$ denote the vector space of all symmetric matrices $X\in\R^{N\times N}$, and let $\aproduct_W$ denote the scalar product on $W$
defined by $\lk X,Y\rk_W := \tr(XY)$, $X,Y\in W$. Furthermore, define the linear mapping $A : \R^M\to W$ by
$$
Ax := T_\Phi^T\diag(x)T_\Phi,\quad x\in\R^M.
$$
By Proposition \ref{p:1st_charac} the frame $\Phi$ is not scalable if and only if there exists no $x\in\R^M$, $x\ge 0$, with $Ax = I_N$. Hence, due to
Lemma \ref{l:farkas}, $\Phi$ is not scalable if and only if there exists $Y\in W$ with $\tr(Y) = \lk I_N,Y\rk_W < 0$ such that
$$
0\le\lk Ae_j,Y\rk_W = \tr((Ae_j)Y) = \tr(\vphi_j\vphi_j^TY) = \vphi_j^TY\vphi_j
$$
holds for all $j=1,\ldots,M$, where $\{e_j\}_{j=1}^M$ denotes the standard basis of $\R^M$. This proves the equivalence of (i) and (ii).

(ii)$\Sra$(iii). For this, let $Y_1\in W$ with $\alpha := -\tr(Y_1) > 0$ such that $\vphi_j^TY_1\vphi_j\ge 0$ for
all $j = 1,\ldots,M$, and set $Y := Y_1 + \frac{\alpha}{N}I_N$. Then $\tr(Y) = 0$ and $\vphi_j^TY\vphi_j > 0$ for all $j = 1,\ldots,M$, as desired.

(iii)$\Sra$(i). Assume now, that there exists $Y\in W$ as in (iii), that is, $\lk I_N,Y\rk_W = 0$ and $\lk Ae_j,Y\rk_W > 0$ for all $j$. Suppose that
$\Phi$ is scalable. Then there exists $x\in\R^M$, $x\ge 0$, such that $Ax = I_N$. This implies
$$
0 = \lk I_N,Y\rk_W = \lk Ax,Y\rk_W = \sum_{j=1}^M x_j\lk Ae_j,Y\rk_W,
$$
which yields $x = 0$, contrary to the assumption $Ax = I_N$. The theorem is proved.
\end{proof}

This theorem can be used to derive a result on the topological structure of the set of non-scalable frames for $\R^N$.
In fact, the corollary we will draw shows that this set is open in the following sense.

\begin{cor}
Let $\Phi = \{\vphi_j\}_{j=1}^M\subset\R^N\setminus\{0\}$ be a frame for $\R^N$ which is not scalable. Then there exists
$\veps > 0$ such that each set of vectors $\{\psi_j\}_{j=1}^M\subset\R^N$ with
\begin{equation}\label{e:eps}
\|\vphi_j - \psi_j\| < \veps\quad\text{ for all }j = 1,\ldots,M
\end{equation}
is a frame for $\R^N$ which is not scalable.
\end{cor}

\begin{proof}
Choosing a subset $J$ of $\{1,\ldots,M\}$ such that $\{\vphi_j\}_{j\in J}$ is a basis of $\R^N$, it follows from the
continuity of the determinant that there exists $\veps_1 > 0$ such that all sets of vectors $\{\psi_j\}_{j=1}^M\subset\R^N$
with \eqref{e:eps} ($\veps$ replaced by $\veps_1$) are frames. By Theorem \ref{t:quadric}, there exists a symmetric matrix
$Y\in\R^{N\times N}$ with $\tr(Y) < 0$ such that $\vphi_i^TY\vphi_i\ge 0$ for all $i$. By adding $\delta I_N$ to $Y$ with
some $\delta > 0$ we may assume without loss of generality that $\tr(Y) < 0$ and $\vphi_j^TY\vphi_j > 0$ for all $j$
(note that the frame vectors of $\Phi$ are assumed to be non-zero). Since the function $x\mapsto x^TYx$ is continuous,
it follows that there exists $\veps\in (0,\veps_1)$ such that for each frame $\{\psi_j\}_{j=1}^M\subset\R^N$ with
\eqref{e:eps} we have $\psi_j^TY\psi_j > 0$ for all $j$. By Theorem \ref{t:quadric}, the frame $\{\psi_j\}_{j=1}^M$ is
not scalable, which finishes the proof.
\end{proof}


\subsection{Geometric Interpretation}\label{subsec:geo}

We now aim to analyze the geometry of the vectors of a non-scalable frame. To derive a precise geometric characterization
of non-scalability, we will in particular exploit Theorem \ref{t:quadric}. As a first step, notice that each of the sets
\[
C_\pm(Y) := \{x\in\R^N : \pm x^TYx > 0\}, \quad Y\in\R^{N\times N} \mbox{ symmetric,}
\]
considered in Theorem \ref{t:quadric} (iii) is in fact an open cone with the additional property that $x\in C_\pm(Y)$
implies $-x\in C_\pm(Y)$. Thus, in the sequel we need to focus our attention on the impact of the condition $\tr(Y) = 0$
on the shape of these cones.

We start by introducing a particular class of conical surfaces, which due to their relation to quadrics -- the exact
relation being revealed below -- are coined `conical zero-trace quadrics'.

\begin{defn}
Let the {\em class of conical zero-trace quadrics} $\calC_N$ be defined as the family of sets
\beq\label{eq:quadricvariety}
\left\{x\in\R^N : \sum_{k=1}^{N-1} a_k \l x,e_k\r^2 = \l x,e_N\r^2\right\},
\eeq
where $\{e_k\}_{k=1}^N$ runs through all orthonormal bases of $\R^N$ and $(a_k)_{k=1}^{N-1}$
runs through all tuples of elements in $\R\setminus\{0\}$ with $\sum_{k=1}^{N-1} a_k = 1$.
\end{defn}

The next example provides some intuition on the geometry of the elements in this class in dimension $N=2,3$.

\newpage

\begin{ex} \label{ex:N23}
{\hspace{2cm}}
\begin{itemize}
\item $N=2$. In this case, by setting $e_\pm := (1/\sqrt 2)(e_1\pm e_2)$, a straightforward computation shows that
$\calC_2$ is the family of sets
\[
\{x\in\R^2 : \l x,e_-\r\l x,e_+\r = 0\},
\]
where $\{e_-,e_+\}$ runs through all orthonormal bases of $\R^2$.
Thus, each set in $\calC_2$ is the boundary surface of a {\it quadrant cone} in $\R^2$, i.e., the union of two orthogonal one-dimensional subspaces in $\R^2$.
\item $N=3$. In this case, it is not difficult to prove that $\calC_2$ is the family of sets
\[
\left\{x\in\R^3 : a\l x,e_1\r^2 + (1-a)\l x,e_2\r^2 = \l x,e_3\r^2\right\},
\]
where $\{e_i\}_{i=1}^3$ runs through all orthonormal bases of $\R^3$ and $a$ runs through all elements in $(0,1)$.
The sets in $\calC_3$ are the boundary surfaces of a particular class of {\it elliptical cones} in $\R^3$.

To analyze the structure of these conical surfaces we let $\{e_1, e_2, e_3\}$ be the standard unit basis and $a\in (0,1)$. Then the quadric
\[
\left\{x\in\R^3 : a\l x,e_1\r^2 + (1-a)\l x,e_2\r^2 = \l x,e_3\r^2\right\}
\]
intersects the planes $\{x_3 = \pm 1\}$ in
\[
\left\{(x_1,x_2,\pm 1) : ax_1^2 + (1-a)x_2^2 =1\right\}.
\]
These two sets are ellipses intersecting the corner points $(\pm 1,\pm 1,\pm 1)$ of the unit cube. Thus, the
considered quadrics are elliptical conical surfaces with their vertex in the origin, characterized by the fact
that they intersect the corners of a rotated unit cube in $\R^3$,
see also Figure \ref{fig:1}(b) and (c).
\end{itemize}
\end{ex}

Note that \eqref{eq:quadricvariety} is by rotation unitarily equivalent to the set
\begin{equation}\label{e:quadric}
\left\{x\in\R^N : x_N^2-\sum_{k=1}^{N-1} a_k x_k^2 = 0\right\}.
\end{equation}
Such surfaces uniquely determine cones by considering their interior or exterior. Similarly, we call the sets
\[
\left\{x\in\R^N : \sum_{k=1}^{N-1} a_k \l x,e_k\r^2 < \l x,e_N\r^2\right\}
\]
\[
\mbox{and} \quad \left\{x\in\R^N : \sum_{k=1}^{N-1} a_k \l x,e_k\r^2 > \l x,e_N\r^2\right\}
\]
the {\em interior} and the {\em exterior} of the conical zero-trace quadric in \eqref{eq:quadricvariety}, respectively.

Armed with this notion, we can now state the result on the geometric characterization of non-scalability.

\begin{thm}\label{t:geometry}
Let $\Phi\subset\R^N\setminus\{0\}$ be a frame for $\R^N$. Then the following conditions are equivalent.
\begin{itemize}
\item[{\rm (i)}]   $\Phi$ is not scalable.
\item[{\rm (ii)}]  All frame vectors of $\Phi$ are contained in the interior of a conical zero-trace quadric of $\calC_N$.
\item[{\rm (iii)}]  All frame vectors of $\Phi$ are contained in the exterior of a conical zero-trace quadric of $\calC_N$.
\end{itemize}
\end{thm}

\begin{proof}
We only prove (i)$\Slra$(ii). The equivalence (i)$\Slra$(iii) can be proved similarly. By Theorem \ref{t:quadric}, a frame
$\Phi = \{\vphi_j\}_{j=1}^M\subset\R^N\setminus\{0\}$ for $\R^N$ is not scalable if and only if there exists a real symmetric
$N\times N$-matrix $Y$ with $\tr(Y) = 0$ such that $\vphi_j^TY\vphi_j > 0$ for all $j=1,\ldots,M$. Equivalently, there exist
an orthogonal matrix $U\in\R^{N\times N}$ and a diagonal matrix $D\in\R^{N\times N}$ with $\tr(D) = 0$ such that
$(U\vphi_j)^TD(U\vphi_j) > 0$ for all $j = 1,\ldots,M$. Note that, due to continuity reasons, the matrix $D$ can be
chosen non-singular, i.e., without zero-entries on the diagonal. Hence, (i) is equivalent to the existence of an orthonormal
basis $\{e_k\}_{k=1}^N$ of $\R^N$ and values $d_1,\ldots,d_N\in\R\setminus\{0\}$ satisfying $\sum_{k=1}^Nd_k = 0$ and
\[
\sum_{k=1}^{N} d_k \l \vphi_j,e_k\r^2 > 0\quad\mbox{for all } j = 1,\ldots,M.
\]
By a permutation of $\{1,\ldots,N\}$ we can achieve that $d_N > 0$. Hence, by setting $a_k := -d_k/d_N$ for $k=1,\ldots,N-1$,
we see that (i) holds if and only if there exist an orthonormal basis $\{e_k\}_{k=1}^N$ of $\R^N$ and
$a_1,\ldots,a_{N-1}\in\R\setminus\{0\}$ such that $\sum_{k=1}^{N-1}a_k = 1$ and
\[
\sum_{k=1}^{N-1} a_k \l \vphi_j,e_k\r^2 < \l\vphi_j,e_N\r^2\quad\mbox{for all } j = 1,\ldots,M.
\]
But this is equivalent to (ii).
\end{proof}

By $\calC_N^*$ we denote the subclass of $\calC_N$ consisting of all zero-trace conical quadrics in which the orthonormal
basis is the standard basis of $\R^N$. That is, the elements of $\calC_N^*$ are quadrics of the form \eqref{e:quadric}
with non-zero $a_k$'s satisfying $\sum_{k=1}^{N-1}a_k = 1$. The next corollary is an immediate consequence of Theorem
\ref{t:geometry} and Corollary \ref{c:unitary}.

\begin{cor}\label{c:geometry}
Let $\Phi\subset\R^N\setminus\{0\}$ be a frame for $\R^N$. Then the following conditions are equivalent.
\begin{itemize}
\item[{\rm (i)}]    $\Phi$ is not scalable.
\item[{\rm (ii)}]   There exists an orthogonal matrix $U\in\R^{N\times N}$ such that all vectors of $U\Phi$ are contained in the
interior of a conical zero-trace quadric of $\calC_N^*$.
\item[{\rm (iii)}]  There exists an orthogonal matrix $U\in\R^{N\times N}$ such that all vectors of $U\Phi$ are contained in the
exterior of a conical zero-trace quadric of $\calC_N^*$.
\end{itemize}
\end{cor}

Utilizing Example \ref{ex:N23}, we can draw the following conclusion from Theorem \ref{t:geometry} for the
cases $N=2, 3$.

\begin{cor}
\bitem
\item[{\rm (i)}] A frame $\Phi\subset\R^2\setminus\{0\}$ for $\R^2$ is not scalable if and only if there exists an open quadrant cone
which contains all frame vectors of $\Phi$.
\item[{\rm (ii)}] A frame $\Phi\subset\R^3\setminus\{0\}$ for $\R^3$ is not scalable if and only if all frame vectors of $\Phi$ are
contained in the interior of an elliptical conical surface with vertex $0$ and intersecting the corners of a rotated unit cube.
\eitem
\end{cor}

To illustrate the geometric characterization, Figure \ref{fig:1} shows sample regions of vectors of a non-scalable frame in $\R^2$
and $\R^3$.

\begin{figure}[ht]\label{fig:1}
\hspace*{1cm}
\includegraphics[width=3cm]{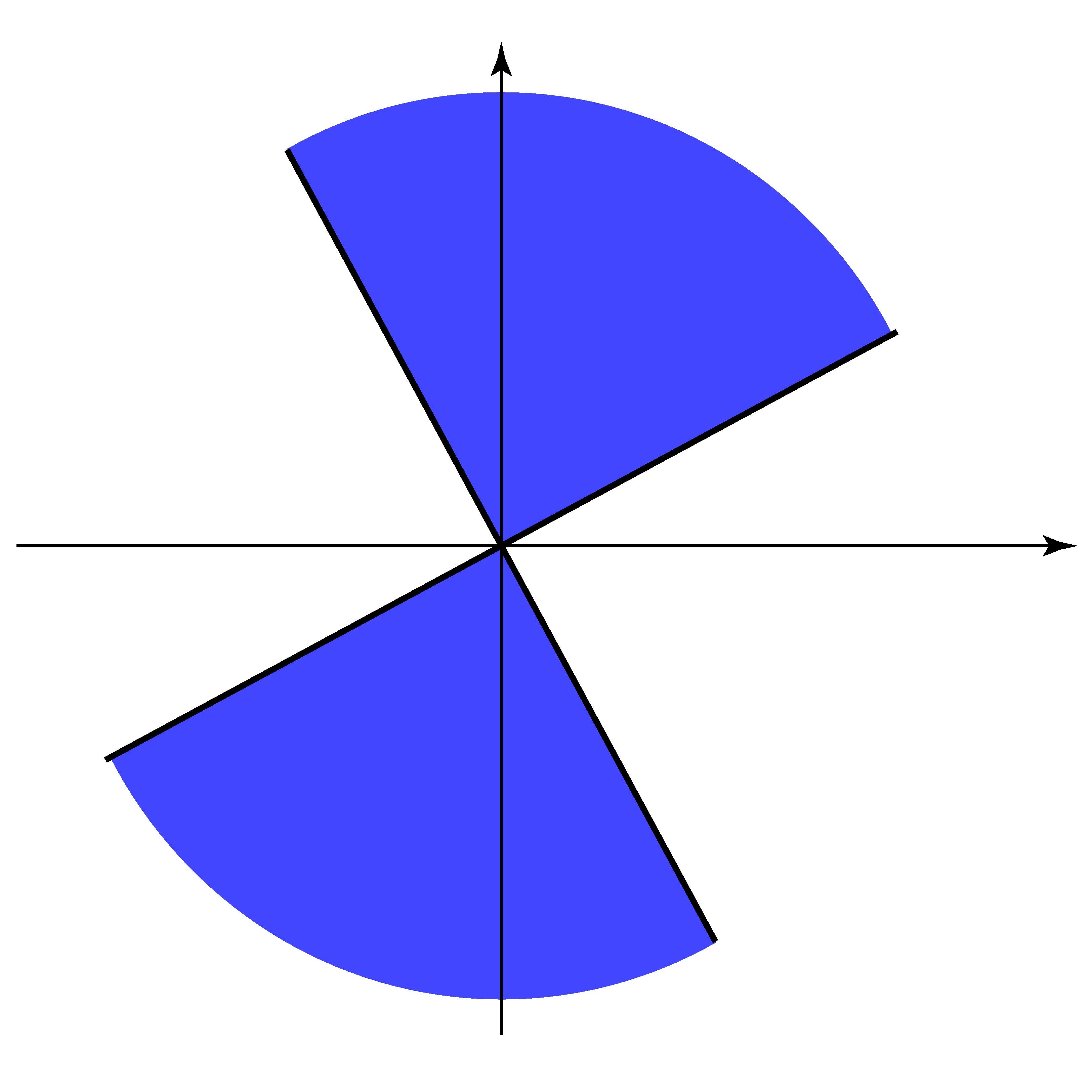}
\hspace*{0.5cm}
\includegraphics[width=4.5cm]{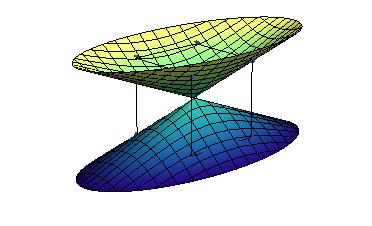}
\includegraphics[width=3.25cm]{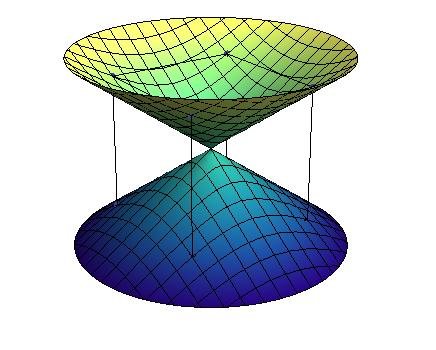}
\put(-300,-15){(a)}
\put(-165,-15){(b)}
\put(-55,-15){(c)}
\caption{(a) shows a sample region of vectors of a non-scalable frame in $\R^2$. (b) and (c)
show examples of sets in $\calC_3$ which determine sample regions in $\R^3$.}
\end{figure}


\section{Acknowledgements}

G.~Kutyniok acknowledges support by the Einstein Foundation Berlin, by Deutsche For\-schungsgemeinschaft
(DFG) Grant SPP-1324 KU 1446/13 and DFG Grant KU 1446/14, and by the DFG Research Center {\sc Matheon}
``Mathematics for key technologies'' in Berlin. F.~Philipp is supported by the DFG Research Center {\sc Matheon}.
K.~A.~Okoudjou  was supported by ONR grants N000140910324 and N0001\-40910144, by a RASA from the Graduate School of
UMCP and by the Alexander von Humboldt foundation. He would also like to express his gratitude to the Institute for
Mathematics at the University of Osnabr\"uck for its hospitality while part of this work was completed.



\section*{Contact information}
Gitta Kutyniok: Technische Universit\"at Berlin, Institut f\"ur Mathematik, Stra\ss e des 17.\ Juni 136, 10623 Berlin, Germany, kutyniok@math.tu-berlin.de

\vspace{0.4cm}\noindent
Kasso A.\ Okoudjou, Department of Mathematics, University of Maryland, College Park, MD 20742 USA, kasso@math.umd.edu

\vspace{0.4cm}\noindent
Friedrich Philipp: Technische Universit\"at Berlin, Institut f\"ur Mathematik, Stra\ss e des 17.\ Juni 136, 10623 Berlin, Germany, philipp@math.tu-berlin.de

\vspace{0.4cm}\noindent
Elizabeth K.\ Tuley, Department of Mathematics, University of California at Los Angeles, Los Angeles, CA 90095 USA, ektuley@umd.edu
\end{document}